\documentclass[oneside,english,reqno]{amsart}

\usepackage{algorithm}
\usepackage{algorithmic}
\usepackage{multirow}

\usepackage{enumitem}

\usepackage[T1]{fontenc}
\usepackage[latin9]{inputenc}
\usepackage{babel}
\usepackage{amsmath}
\usepackage{amsthm}
\usepackage{amssymb}

\usepackage{comment}

\usepackage{graphicx}

\usepackage{graphicx}
\usepackage{tikz}
\usetikzlibrary{shapes.misc}
\usepackage{float}

\usepackage[unicode=true,pdfusetitle,
bookmarks=true,bookmarksnumbered=false,bookmarksopen=false,
breaklinks=false,pdfborder={0 0 1},backref=false,hidelinks]
{hyperref}

\makeatletter

\numberwithin{equation}{section}
\numberwithin{figure}{section}

\makeatother

\newtheorem{theorem}{Theorem}

\newtheorem{definition}{Definition}
\newtheorem{example}[theorem]{Example}

\newtheorem{remark}{Remark}

\newtheorem{assumption}{Assumption}

\begin{document}

\title[On global solvability of a class of possibly nonconvex QCQP in Hilbert]{On global solvability of a class of possibly nonconvex QCQP problems in Hilbert spaces
}

\author{
	Ewa M. Bednarczuk$^{1}\,^{2}$
}
\author{
	Giovanni Bruccola$^1$
}

\thanks{$^1$ Systems Research Institute, Polish Academy of Sciences, Newelska 6, 01-447 Warsaw}
\thanks{$^2$  Warsaw University of Technology, Koszykowa 75, 00-662 Warsaw}
\email{Ewa.Bednarczuk@ibspan.waw.pl} 
\email{Giovanni.Bruccola@ibspan.waw.pl}

\begin{abstract}
 We provide conditions ensuring that the KKT-type conditions characterizes the global optimality for
 quadratically constrained (possibly nonconvex) quadratic programming QCQP problems in Hilbert space. The key property is the convexity ofa  image-type set related to the functions appearing in the formulation of the problem. 
 The proof of the main result relies on a generalized version of the (Jakubovich) S-Lemma in Hilbert spaces.
 As an application, we consider the class of QCQP problems with a special form of the quadratic terms of the constraints.
\end{abstract}
\maketitle

\section{Introduction}
The aim of this work is to establish a sufficient condition under which the KKT-type conditions  characterizes the global optimality of a possibly non convex Quadratically Constrained Quadratic Programming (QCQP) problem in Hilbert spaces. 

Let $H$ be a Hilbert space with the inner product 
$\langle a, b \rangle$, for $a,b\in H$,  and the associated norm  $\|\cdot\|=\sqrt{\langle \cdot, \cdot \rangle}$. 

A generic QCQP problem is defined as follows.
\begin{equation}
    \label{QCQP1}
    \tag{QCQP}
    \begin{split}
        &{Minimize\,}_{x\in H} \ \ J(x):=\langle x, A_Jx\rangle+2\langle b_J, x\rangle+c_J\\
        &s.t.\ \ f_k(x):=\langle x, A_kx\rangle+2\langle b_k, x\rangle+c_k \le 0,\ \ k=1,...,m
    \end{split}
\end{equation}
 where 
  the continuous linear self-adjoint operators $A_J,A_k: H\rightarrow H$,   $b_J,b_k \in H $,  and $c_J,c_k \in \mathbb{R}$ are given data for $k=1,...,m$. 

Problems of the form QCQP (constrained and unconstrained)  appear in many contexts. For a recent application in succesive quadratic approximation method, see \cite{lee2019inexact}.
In $\mathbb{R}^n$,
QCQP problems also appear in the discretization of ill-posed problems,  \cite{golub1999tikhonov}, maximum clique problem, \cite{hosseinian2018nonconvex}, the circular packing problem, \cite{stetsyuk2016global} and the Chebyshev center problem, \cite{xia2021chebyshev}. Regarding Hilbert spaces, quadratic problems arise in calculus of variations, \cite{hestenes1951applications}, in optimal control, \cite{yakubovich1992nonconvex} and in in the context of reproducing kernel
Hilbert space, \cite{signoretto2008quadratically}.

We consider both general \ref{QCQP1} problems and  a special class of \ref{QCQP1} problems which satisfy the following assumption:
\begin{assumption}
\label{assumaI}
Let $A_J,A_{k}$ $\forall\,k\in\{1,...,m\}$ be of the form $A_J=a_JI,A_k=a_kI$, where $a_J,a_k\in\mathbb{R}$ and $I$ is the identity operator.
\end{assumption}

When \autoref{assumaI} holds, \ref{QCQP1} takes the form
\begin{equation}
    \label{QCQP2}
    \tag{S-QCQP}
    \begin{split}
        &{Minimize\,}_{x\in H} \ \ J(x):=a_J\|x\|^2+2\langle b_J,x\rangle+c_J\\
        &s.t.\ \ f_k(x):=a_k\|x\|^2+2\langle b_k,x\rangle+c_k \le 0,\ \ k=1,...,m
    \end{split}
\end{equation}

Since in \ref{QCQP2} the operators $A_J,A_k$, $A_J,A_k$ $k=1,...,m$ do not appear and instead we only have the scalars $a_J,a_k\in\mathbb{R}$, $k=1,...,m$, in the sequel we refer to \ref{QCQP2} as the \textit{scalar QCQP}.

\begin{remark}
    The scalars $a_J,a_k$, $k=1,...,m$ can be equal to zero for \ref{QCQP2}.
    When $a_J=0$ and/or $a_k=0$ for some $k$, the corresponding $J(x)$ or $f_k(x)$ are linear forms. In case $a_J=0$ and/or $a_k=0$ for all $k$, \ref{QCQP2}  become a linear or a quadratic convex optimization problem, respectively.
\qed
\end{remark}

\begin{definition}
Given a self-adjoint linear continous operator $A$ (i.e. $\langle Ax,x\rangle=\langle x,Ax\rangle$), the quadratic form ${A(x):=\langle x, Ax\rangle}$ is said to be positive if $\langle x, Ax\rangle>0$, $\forall\,x\in H\setminus \{0\}$ and non negative if $\langle x, Ax\rangle\ge0$, $\forall\,x\in H$.\qed
\end{definition}

By Proposition 3.71 of \cite{bonnans2013perturbation}, a quadratic form $\langle x, Ax\rangle+2\langle b,x\rangle +c$ is convex if and only if $A(x):=\langle x, Ax\rangle$ is non negative. 

We say that KKT conditions for QCQP are satisfied at a feasible point $x^{*}$ if there exists $\gamma=(\gamma_{1},...,\gamma_{m})$, $\gamma_{i}\ge 0$, $m=1,...,m$ such that
\begin{equation}
    \label{kkt0}
    \tag{KKT}
    \begin{split}
        &(i)\ \ \ \nabla(J+\sum\limits_{k=1}^m\gamma_k f_k)(x^*)=0, \text{stationarity} \\
        &(ii)\ \ \gamma_kf_k(x^*)=0\ \ k\in\{1,...,m\}, \text{complementarity} \\
        &(iii)\ (A_J+\sum\limits_{k=1}^m\gamma_kA_k)(x) \text{ non negative}.
    \end{split}
\end{equation}

Then, the optimal value of \eqref{QCQP1} is $J^*=J(x^*)$.

In general, under standard constraint qualification ( e.g. Mangasarian-Fromowitz CQ) KKT conditions are necessary for local optimality of $x^{*}$ for QCQP, see Theorem 3.6 of \cite{van2021optimality}.

In this paper, we determine conditions under which \ref{kkt0} are necessary and sufficient global optimality conditions in Hilbert spaces for \ref{QCQP1}, section 3, and \ref{QCQP2}, section 4.

 The authors of \cite{jeya1} prove that \ref{kkt0} are necessary and sufficient optimality conditions in $\mathbb{R}^n$ for \ref{QCQP1} problems if the matrices $H_J:=\begin{pmatrix}
A_J&b_J\\b_J^T&c_J
\end{pmatrix}$ and $H_k:=\begin{pmatrix}
A_k&b_k\\b_k^T&c_k
\end{pmatrix} $  $\forall\,k\in\{1,...,m\}$ are $Z$-matrices. $Z$-matrices are matrices with non positive off diagonal elements. 
In order to characterize the optimal value of Z-matrices QCQP, in \cite{jeya1} the authors propose a generalized version of the S-Lemma for $Z$-matrices QCQP, in the sense described in Section 2.

We apply the approach similar  to that proposed in \cite{jeya1} 
and we base our developments on a generalized version of the S-Lemma in Hilbert spaces.
The key property is the convexity of the \textit{Generalized Image Set} 
\begin{equation} \tag{GIS}
\label{GIS}
\Omega_0:=\{(f_0(x),f_1(x),...,f_m(x))\ |\ x\in H\}+ int\mathbb{R}^{m+1}_+,
\end{equation} $f_0:=J(x)-J^*$ and $'+'$ is the Minkowski sum of the range $\{(f_0(x),f_1(x),...,f_m(x))\ |\ x\in H\}$ and the interior of the cone $\mathbb{R}^{m+1}_+$. By \cite{holmes2012geometric}, exercise 2.1, the set $\Omega_{0}$ is open. 
Recall, that the image set of optimization problems defined as 
\begin{equation}
    \label{eq: image set gianessi}
    \{(f_0(x),f_1(x),...,f_m(x))\ |\ x\in H\}
\end{equation}
has been extensively investigated in the monograph by \cite{giannessi2006constrained}. 
In \cite{jeya1}, it is shown that \eqref{eq: image set gianessi} is not sufficient to establish if the \ref{kkt0} conditions are necessary and sufficient for global optimality of every class of \ref{QCQP1}. They consider instead \ref{GIS}.

Showing the convexity of \cite{giannessi2006constrained} or of \ref{GIS} is a key step in the proof of the famous Jakubovich S-Lemma, \cite{yakubovich1992nonconvex}.

The S-Lemma is an important result related to the  general S-procedure described in \cite{fradkov1},
treated from an historical point view in \cite{gusev1}. Paper \cite{polik1} provides a comprehensive survey on the S-Lemma in $\mathbb{R}^n$, while \cite{jeya2} explores the relations between the S-Lemma and the Lagrangian multipliers of \ref{QCQP1}.
In \cite{barro1}, an S-Lemma based approach similar to the one proposed by \cite{jeya1} and the present paper  is applied to problem with data uncertainty in $ \mathbb{R}^n$.
The paper \cite{RZ1} extend the results of \cite{jeya1} to Z-matrices \ref{QCQP1} in $ \mathbb{R}^n$ with infinite number of inequalities.
More recent generalization of the S Lemma in $\mathbb{R}^n$ appear in \cite{xia2016s}, \cite{song2022calabi}. Also in these latter papers, the sets  \eqref{eq: image set gianessi} and \ref{GIS} play an important role.

 In $\mathbb{R}^n$, the convexity of the set $\Omega_0$ for \eqref{QCQP2} is proved in \cite{beck2007convexity}, when $m+1\le n$.
For \eqref{QCQP1} with one or two constraints, i.e. $m=1$ and $m=2$, the convexity of $\Omega_0$ can be proved under some regularity conditions, starting from the results of Dine and Polyak (see \cite{jeya1}, \cite{polik1}).
For  Z-matrices \ref{QCQP1}, the convexity of $\Omega_0$ is proved in \cite{jeya1}. 

In \cite{continopolyak}, the results of  Dine and Polyak are generalized in Hilbert spaces. As a conseguence, a generalized form of the S-Lemma for three homogeneous quadratic functionals, i.e. $f_k(x):=\langle x, A_kx\rangle$, $k=0,1,2$, holds in Hilbert spaces. 

Other results concerning generalized S-Lemma in Hilbert spaces for homogeneous quadratic functionals can be found in \cite{yakubovich1992nonconvex}.

In section 3, we prove that, for general \eqref{QCQP1} in Hilbert space, if $\Omega_0$ is convex, then the generalized S-Lemma holds and the \eqref{kkt0} conditions are necessary and sufficient for global optimality.

Throughout the whole paper, we are making the following existence assumption.
\begin{assumption}
\label{preliminaryassumpt}
There exists a global minimum $x^*$ of problem \ref{QCQP1} and there exists $\gamma\in\mathbb{R}^m_+\backslash\ 0$ such that $(A_J+\sum\limits\gamma_kA_k)(x)$ is non negative.
\end{assumption}

The assumption $\exists\,\gamma\in\mathbb{R}^m_+\backslash\ 0$ such that the matrix $A_J+\sum\limits\gamma_kA_k\succeq0$, is a standard assumption in the literature related to finite dimensional QCQP, see e.g. \cite{loc0}, \cite{Boyd1} and \cite{blek1}. Observe that it is non verified by nonconvex QP (quadratic problems with linear constraints only), since $A_J\not\succeq0$. 
KKT conditions and the SDP relaxation for QP in $\mathbb{R}^n$ are studied e.g. in \cite{gon1}, \cite{burer2}.

The assumption that there exists a global minimum $x^*$ of problem \ref{QCQP1}, it is made to avoid discussing the existence of solution for \ref{QCQP1} problems in Hilbert spaces, which is covered by other papers (see e.g. \cite{dong2018solution}). 
When we will explicitly calculate the KKT conditions for \ref{QCQP2}, we will assume the objective function to be $J(x):=\|x-z\|^2$, $z\in H$ to guarantee the existence of a solution on a closed, but possibly non convex, constrained set. 

For convex QP problems (quadratic convex objective functions under linear constraints only) in Hilbert spaces, conditions for the existence of solutions and global optimality conditions can be found respectively in Theorem 3.128 and Theorem 3.130 of \cite{bonnans2013perturbation}.



The organization of the paper is as follows.
In Section 2, we provide the preliminaries concerning the notation, the S-Lemma, the Fermat rule and the convex separation theorem which are used in the sequel.

The main result of Section 3 is \autoref{globalminchar}, which provides the global minima characterization for general \ref{QCQP1} problems in the form of \ref{kkt} conditions, under the assumption that $\Omega_0$ is convex.

In Section 4, we apply \autoref{globalminchar} to provide a characterization of the global solution for \ref{QCQP2}.
The main result is \autoref{omega0isconvex} which proves the convexity of the set $\Omega_0$ for this class of QCQP problems.

\section{Preliminaries}

Given a vector $x\in \mathbb{R}^n$, we say $x\ge0$ when every component is non-negative and we define ${\mathbb{R}^n_{+}:=\{x\in\mathbb{R}^{n}\ |\ x\ge 0\}}$. 
$0$ may also denotes the all-zero vector in $H$.
Given two vectors $x,y\in H$, $\langle x,y\rangle$ denotes the inner product between $x$ and $y$. When $x=\mathbb{R}^n$, we have $\langle x,y\rangle=x^Ty=y^Tx$.
The corresponding norm is denoted by $\|\cdot\|=\sqrt{\langle\cdot,\cdot\rangle}$.

Let $S^n$ be the set of symmetric matrices in $\mathbb{R}^{n\times n}$.
$S^n_+$ and $S^n_{++}$ are the cones of symmetric matrices which are also positive semidefinite and positive definite, respectively. 
If a matrix $A$ belongs to $S^n_+$ then we write $A\succeq0$; if $A\in S^n_{++}$ then we write $A\succ0$.

By $diag(A)\in\mathbb{R}^n$, we denote the vector of the elements in the main diagonal of $A\in\ S^n$.
The \textit{trace inner product} $\langle A, B\rangle$,
between  symmetric matrices $A$, $B$  of dimension $n\times n$ is defined as $\langle A, B\rangle:=Tr(B^{T}A)=\sum_{i=1}^{n}\sum_{j=1}^{n}a_{ij}b_{ij}$.
 Let $|\cdot|$ denote the cardinality of a set.
 $C$ is an affine subspace if $C\neq\emptyset$ and $\forall\,\lambda\in\mathbb{R}$ and for all distinct $x,y\in C$
 \begin{equation*}
     \label{affine}
     \lambda x + (1-\lambda) y \in C
 \end{equation*}
 $\textit{aff}\, C$ denotes the affine hull of $C$, i.e. the smallest affine subspace of $H$ containing $C$, 
 $int\,C$ denotes the interior of $C$, 
 $$
int\,C:=\{x\in C\,|\,(\exists\,\epsilon>0)\,(x+\epsilon B_1)\subset C\}.
 $$
 $B_1:=\{x\,|\,\|x\|\le1\}$ be the \textit{unit ball} in $H$.
 The relative interior of $C$, denoted $ri\,C$, can be espressed in $\mathbb{R}^n$ as :
 \begin{equation}
     \label{riC}
     ri\,C:=\{x\in \textit{aff}\,C\,|\,(\exists\,\epsilon>0)\,(x+\epsilon B_1)\cap(\textit{aff}\,C)\subset C\}.
 \end{equation}

\subsection{Jakubovich S-Lemma, basic separation theorem}  
An important theorem for the optimality conditions of \ref{QCQP1} is the S-lemma.
We recall a generalized version of the S-Lemma that can be found in \cite{polik1}.

\begin{theorem} (Jakubovich  S-Lemma)
\label{Slemma}
Let $f,g: \mathbb{R}^{n} \to \mathbb{R}$ be quadratic functionals and suppose that there exists a point $x_0\in \mathbb{R}$ such that $g(x_0)< 0$. The following statements (i) and (ii) are equivalent.
\begin{align*}
    &(i)\ \ 
    (\nexists\, x \in \mathbb{R}^n)\ \ s.t.\ \
    f(x) < 0 \ \ g(x) \le 0
   \\
    &(ii)\ \ \exists \gamma \ge 0 \,\,\,such\,\,\,that\,\,\, f(x)+\gamma g(x) \ge 0 \,\,\, \forall x \in \mathbb{R}^{n}
\end{align*}



\end{theorem}

Consider a collection of quadratic functionals ${f_k: H \to \mathbb{R}}$ ${(k=1,...,m)}$.
A theorem of the alternative is called generalized version of the S-Lemma if it establishes the conditions on the functionals $f_k$ under which only one between the following statements holds:
\begin{enumerate}
    \item $\exists\,x\in H$ such that $ f_k(x)<0\ \ \forall\,k\in\{1,...,m\}$
    \item $(\exists\,\gamma\in\mathbb{R}_+^m\backslash0)$ $\sum\limits_{k=1}^m\gamma_kf_k(x)\ge0$ $\forall\,x\in H$
\end{enumerate}

Let $\partial\,f(x)$ denote the convex subdifferential of $f$ at $x$, then we write $v\in\partial\,f(x)$ if
\begin{equation}
    \label{eq:subdifferential}
    f(y)-f(x)\ge \langle v, y-x\rangle \quad \forall\,y\in H.
\end{equation}

If $f$ is differentiable and convex, $\partial\,f(x)=\{\nabla f(x)\}$. By \eqref{eq:subdifferential}, it is possible to prove the Fermat's optimality condition, see e.g. Theorem 16.3 of \cite{bauschke2017convex}.
\begin{theorem}
\label{fermat}
(Fermat optimality conditions)
Let $f:H\rightarrow\ ]-\infty,+\infty]$ be proper. Then 
$$
Arg\,min\,f=\{x\in H\ |\ 0\in\partial f(x)\}.
$$
\end{theorem}

Let $C_1$ and $C_2$ be non-empty subsets in $\mathbb{R}^n$.

\begin{definition}
\label{hypersep}
(\cite{rocka}, section 11)
\begin{itemize}
    \item A hyperplane $P$ is said to \textit{separate} $C_1$ and $C_2$ if $C_1$ is contained in one of the closed half spaces associated to $P$ and $C_2$ is contained in the opposite closed half space.
\item $P$ is said to \textit{separate properly} $C_1$ and $C_2$ if they are not both contained in $P$ itself.
\end{itemize}
\end{definition}

We are ready to state the convex separation theorem in finite dimensions, that will be crucial in the next section.

\begin{theorem}
\label{convexsept}
(\cite{rocka}, Theorem 11.3)

Let $C_1$ and $C_2$ be non-empty convex sets in $\mathbb{R}^n$. 
In order that there exists a hyperplane that separates $C_1$ and $C_2$ properly, it is necessary and sufficient that $ri\,C_1$ and $ri\,C_2$ have no point in common. 
\end{theorem}

For $n$-dimensional convex sets in $\mathbb{R}^n$, $\textit{aff}\,C=\mathbb{R}^{n}$  and so, by \eqref{riC}, we have $ri\,C=int\,C$ (\cite{rocka}, section 6).

Hence, we can rewrite \autoref{convexsept} as follows.
\begin{theorem}
\label{convexsepth}

Let $C_1$ and $C_2$ be $n$-dimensional non-empty convex sets in $\mathbb{R}^n$. 
In order that there exists a hyperplane that separates $C_1$ and $C_2$ properly, it is necessary and sufficient that $int\,C_1$ and $int\,C_2$ have no point in common. 
\end{theorem}

\subsection{ Direct sum decompositions in Hilbert spaces}
We apply the following concepts and results in the fourth section of the present work. This subsection is based on the monograph \cite{deutsch2001best}.

If $H_1$ and $H_2$ are subspaces of a { Hilbert space   $H$} 
we write ${H=H_1 \oplus H_2}$
if each $x \in H$ has a unique representation in the form $x = x_1 + x_2$, 
where $x_1\in H_1$ and $x_2\in H_2$. 
${H_1 \oplus H_2}$ is the direct sum of $H_{1}$ and $H_{2}$. 
If $H_1$ is a closed subspace of $H$, one can always find a subspace $H_2$ such that ${H=H_1 \oplus H_2}$, $H_2$  is called a complement to $H_1$ 
and is orthogonal to $H_1$,  $H_2=H_1^{\perp}$, i.e. $\langle x_{1}, x_{2}\rangle =0$, for any $x_{1}\in H_1$, $x_{2}\in H_2$.

\begin{definition}
\label{def:Schauderbasis}
Given a finite dimensional subset $H_1$ of a Hilbert space $H$, a basis of $H_1$ is a set of maximal linearly indipendent vectors $y_k\in H_1$, $k=1,2,..,m$, $m\in \mathbb{N}$ such that every vector $y$ of $H_1$ can be represented as $\sum\limits_{k=1}^{m}\alpha_ky_k$. We say that the set of vectors $y_k\in H$, $k=1,2,..,m$, $m\in \mathbb{N}$ spans $H_1$. 
The dimension of a subspace $H_1$ of $H$, denoted $\text{dim }H_1$, is equal the number of vectors in a basis of  $H_1$. 
\end{definition}

\begin{definition}
\label{def:codim}(Definition 7.9 from \cite{deutsch2001best})
Let $H$ be a inner product space.
A closed subspace $H_1$ of $H$ is said to have codimension $m<\infty$, 
written ${\text{codim }H_1 = m}$, if there exists a subspace $H_2$ with $\text{dim }H_2=m$ such that $H=H_1 \oplus H_2$. 
\end{definition}

The following results is a consequence of Theorem 7.11 and Theorem 7.12 from \cite{deutsch2001best}.
\begin{theorem}
    \label{th:deu712}
    Let $H$ be a Hilbert space.
   Let $V$ be a closed subspace of  $H$ and $m$ be a positive integer. Then, $\text{codim }V=m$ if and only if 
$$V = \bigcap\limits_{i=1}^m\{x \in H\ |\ \langle b_i,x\rangle = 0\},$$ 
for some linearly independent set $\{b_l,b_2, \cdots,b_m \}$ in $H$.
\end{theorem}

We end this section with the following theorem on linear operators in Hibert spaces. {Let us recall that, for a given linear continuous operator $A : H_1 \rightarrow H_2$ acting between two Hilbert spaces $H_{1}$ and $H_{2}$,
$$ 
\text{range}\, A:=\{y\in H_{2}\ | \  y=Ax\ \text{for some}\ x\in H_{1} \}, \ \ \text{ker }A:=\{x\in H\ |\ Ax=0\}.
$$
Let $\langle \cdot,\cdot\rangle_1$ and $\langle \cdot,\cdot\rangle_2$ be the inner products associated respectively to $H_1$ and $H_2$. 
The adjoint operator $A^{*}:H_{2}\rightarrow H_{1}$ is a continuous linear operator such that
$$
\langle Ax, y\rangle_{2}=\langle x\ A^{*}y\rangle_{1}.
$$
When $H_{1}$ and $H_{2}$ are finite-dimensional spaces,  and the operator $A$ is represented by a matrix $A$ then the adoint operator is represented by the transposed matrix $A^{T}$.
}
\begin{theorem}
\label{th:linop}
Let $H_1$, $H_2$ be Hilbert spaces equipped respectively with inner products $\langle\cdot,\cdot\rangle_{1}$ and $\langle\cdot,\cdot\rangle_{2}$ . If $A : H_1 \rightarrow H_2$ is a continuous linear operator, then
\begin{equation}
    \label{eq:linop}
    \overline{\text{range}}\,A = (\text{ker }A^*)^{\perp},\qquad \text{ker }A=(\text{range }A)^{\perp}
\end{equation}
where $\overline{C}$ denotes the closure of the set $C$, and ${C^{\perp}=\{ x\in H\ |\ \langle x,c\rangle=0,\,\forall\,c\in C\}}$.
\end{theorem}
\begin{proof}

Let $x \in \text{range }A \subset H_2$. There exists $y \in H_1$ such that $x = Ay$.
For any $z \in \text{ker } A^*\subset H_2$ we have
\begin{equation*}
    \langle x, z\rangle_2  = \langle Ay, z\rangle_2 = \langle y, A^*z\rangle_1.
\end{equation*}
This proves that $\text{range }A \subset (\text{ker } A^*)^{\perp}$. Since $(\text{ker } A^*)^{\perp}$ is closed, it follows that $\overline{\text{range}}\,A \subset (\text{ker } A^*)^{\perp}$.
On the other hand, if $x \in (\text{range }A)^{\perp} \subset H_2$ , then for all $y \in H_1$ we
have
$$0 = \langle Ay, x\rangle_2 = \langle y, A^*x\rangle_1,$$
i.e. $A^*x = 0$. This means that $(\text{range }A)^{\perp}\subset  \text{ker } A^*$.
By taking the orthogonal complement of this relation, we get
\begin{equation*}
    (\text{ker } A^*)^{\perp}\subset (\text{range }A)^{\perp\perp}=\overline{\text{range}}\,A
\end{equation*}

which proves the first part of \eqref{eq:linop}. To prove the second part, we apply
the first part to $A^*$
instead of $A$, use $A^{**} = A$ and take the orthogonal
complement.
\end{proof}
An equivalent formulation of this theorem is that for any continuous
linear operator $A : H_1 \rightarrow H_2$ we have
\begin{equation}
    \label{eq:range+ker}
    H_2=\overline{\text{range}}\,A\oplus\text{ker }A^*,\quad H_1=\text{ker }A\oplus\text{range }A^*
\end{equation}

\section{Global Minima Characterization for general (QCQP)}
\label{global:characterisation}

In this section we characterize global minima of  \ref{QCQP1} problem by  \eqref{kkt} conditions derived with the help of a generalized form of the S-Lemma as defined in Section 2.

Our approach is inspired by the one proposed in \cite{jeya1} to characterize, in $\mathbb{R}^n$, the global minima of $Z$-matrices \ref{QCQP1}, i.e. \ref{QCQP1} with the matrices
\begin{equation}
\label{matrix:H}
H_k:=\begin{pmatrix}A_k&b_k\\b_k^T&c_k\end{pmatrix}\ \ k=1,...,m\text{ and }H_J:=\begin{pmatrix}A_J&b_J\\b_J^T&c_J\end{pmatrix}
\end{equation}
having all the off diagonal elements non positive.

Consider a collection of quadratic functionals $f_{k}:H\rightarrow\mathbb{R}$, $k=0,1,...,m$ defined on a Hilbert space $H$,
\begin{equation} 
\label{eq:functionals}
f_{k}(x):=\langle x, A_{k}x\rangle +\langle b_{k},x\rangle+c_{k},
\end{equation} 
where $A_{k}:H\rightarrow H$ are self-adjoint linear continuous operators acting on the space $H$ and $b_{k}\in H$, $k=0, 1,...,m$.
\begin{theorem}
 \label{th:shilbert}   
 Let $H$ be a Hilbert space and let $f_{k}:H\rightarrow\mathbb{R}$, $k=0,1,...,m$  be a finite collection of functionals of the form \eqref{eq:functionals}. Consider the set \ref{GIS}, defined as
 \begin{equation*}
    \Omega_0:=\{(f_0(x),f_1(x),...,f_m(x))\ |\ x\in H\}+ int\mathbb{R}^{m+1}_+.
\end{equation*}
 Assume that the set \ref{GIS} is convex. 
 Then, exactly one of the following statements is valid:
\begin{enumerate}
    \item[{(i)}] $\exists\,\bar{x}\in H$ such that $ f_k(\bar{x})<0\ \ \forall\,k\in\{0,...,m\}$
    \item[{(ii)}] $(\exists\,\gamma\in\mathbb{R}_+^{m+1}\backslash\textbf{0}_{m+1})$ $\sum\limits_{k=0}^m\gamma_kf_k(x)\ge0$ $\forall\,x\in H$
\end{enumerate} 
\end{theorem}
\begin{proof}
The implication [not(ii)$\Rightarrow$(i)] is immediate (by contradiction).
To show the implication [not(i)$\Rightarrow$(ii)], assume that (i) does not hold, i.e., the system 
\begin{equation}
\label{slater1h}
f_k(x)<0\ \ \forall\,k\in\{0,...,m\}
\end{equation}
has no solution.
By the definition of \ref{GIS}, the inconsistency of the system \eqref{slater1h} implies that 
\begin{equation} 
\label{separationzeroh}
\Omega_0\cap\ (-\text{int} \mathbb{R}_+^{m+1})=\emptyset.
\end{equation}
To see this, suppose by contrary,  that there exists $y\in\Omega_{0}\cap (-\text{int}\mathbb{R}_{+}^{m+1})$. By the definition of $\Omega_{0}$ (with $f_{k}$, $k=0,...,m$, defined by \eqref{eq:functionals}), there exist $x_{0}\in H$,  $C_{0}, C_{1}\in\text{int}\mathbb{R}_{+}^{m+1}$ such that
$$
y_{0}=(f_{0}(x_{0}),...,f_{m}(x_{0}))+C_{0}=-C_{1}\in
(-\text{int}\mathbb{R}_{+}^{m+1}),
$$
i.e., $(f_{0}(x_{0}),...,f_{m}(x_{0}))=-C_{0}-C_{1}\in(-\text{int}\mathbb{R}_{+}^{m+1})$ contradictory to \eqref{slater1h}. This proves \eqref{separationzeroh}.

 Since $\text{int}\Omega_0$ and $-\text{int}\mathbb{R}_{+}^{m+1}$ are $(m+1)$-dimensional sets in $\mathbb{R}^{m+1}$, non-empty and convex, by \eqref{separationzeroh}, 
$$ 
int\,\Omega_0\cap\ \,(-\text{int} \mathbb{R}_+^{m+1})=\emptyset,
$$
 we can apply \autoref{convexsepth}. 
 So, there exists a hyperplane which separates $\Omega_0$ and $-\text{int}\mathbb{R}_{+}^{m+1}$ properly, i.e. there exists $\gamma=(\gamma_0,...,\gamma_m)\in\mathbb{R}^{m+1}\backslash 0$ such that
\begin{equation} 
\label{separationh}
\sum_{k=0}^{m} \gamma_{k}y_{k}\ge 0\ \ \forall\ y=
(y_{0},...,y_{m})\in\Omega_{0}
\end{equation}
and $\sum_{k=0}^{m} \gamma_{k}z_{k}\le 0\ \ \forall\ z=(z_{0},...,z_{m})\in(-int \mathbb{R}_+^{m+1}).$ This latter inequality implies that it must be $\gamma\in\mathbb{R}_{+}^{m+1}\backslash 0$.

Consequently, by the definition of  $\Omega_{0}$, 
$$
y=(f_{0}(x),...,f_{m}(x))+C\in\Omega_{0},\ \  C=(C _{0},...,C_{m})\in\text{int} \mathbb{R}_{+}^{m+1},\ \ x\in H
$$
and the formula \eqref{separationh}, we get 
\begin{equation} 
\label{separation1h}
\sum_{k=0}^{m} \gamma_{k}(f_{k}(x)+C_{k})\ge 0\ \ \forall\ x\in H, \ \ C_k\in\text{int}\mathbb{R}_+\ k=0,...,m.
\end{equation}
We show that  it must be
\begin{equation} 
\label{separation2h}
\sum_{k=0}^{m} \gamma_{k}f_{k}(x)\ge 0\ \ \forall\ x\in H.
\end{equation}
Otherwise, 
$
\sum_{k=0}^{m} \gamma_{k}f_{k}(\bar{x})< 0\ \ \text{for some  }\ \bar{x}\in H,$ and it would be possible to choose $C\in\text{int}\mathbb{R}_{+}^{m+1}$ with components $C_{k}>0$ small enough so as 
$$
\sum_{k=0}^{m} \gamma_{k}(f_{k}(\bar{x})+C_{k})< 0
$$
which would contradict \eqref{separationh} since $(f_{0}(\bar{x})+\ell_{0},...,f_{m}(\bar{x})+\ell_{m})\in\Omega_{0}$ for any $\ell=(\ell_{0},...,\ell_{m})\in \text{int}\mathbb{R}^{m+1}_+\ $.
Thus, \eqref{separation2h} holds and it proves $(ii)$.
\end{proof}

\renewcommand{\theenumi}{\roman{enumi}}%

In the following, we exploit \autoref{th:shilbert} to get necessary and sufficient optimality conditions for general \ref{QCQP1}.

\begin{theorem}
\label{globalminchar}
Let \autoref{preliminaryassumpt} hold and
$x^*$ be a global minimizer of \ref{QCQP1}.
Define
 $$
 f_0(x):=J(x)-J(x^*)=\langle x, A_Jx\rangle+\langle b_J,x\rangle-\langle x^*, A_Jx^*\rangle-2\langle b_J, x^*\rangle.
 $$ 
 Consider the collection of quadratic functionals formed by $f_0(x)$ and the constraints of \eqref{QCQP1} $f_i(x)$, $i=1,...,m$.
Denote \ref{GIS} with $\Omega_0$ and let $\Omega_0$ be convex.
The following Fritz-John conditions are necessary for optimality, i.e.
there exists a vector $(\gamma_0,...,\gamma_m)\in\mathbb{R}^{m+1}_+\backslash\ 0$ such that
\begin{equation}
    \label{fritzjohn}
    \begin{split}
        &(i)\ \ \ \nabla(\gamma_0J+\sum\limits_{k=1}^m\gamma_k f_k)(x^*)=0,\\
        &(ii)\ \ \gamma_kf_k(x^*)=0\ \ k\in\{1,...,m\},\\
        &(iii)\ \gamma_{0}A_J+\sum\limits_{k=1}^m\gamma_kA_k \text{ is non negative}.
    \end{split}
\end{equation}
Moreover, if there exists a point $x_0\in H$ such that \begin{equation} \label{slater}
f_k(x_0)<0\ \ \ \forall\,k\in\{1,...,m\},
\end{equation}
then 
there exists a vector $(\gamma_1,...,\gamma_m)\in\mathbb{R}^m_+\backslash0$ such that
\begin{equation}
    \label{kkt}
    \tag{KKT}
    \begin{split}
        &(i)\ \ \ \nabla(J+\sum\limits_{k=1}^m\gamma_k f_k)(x^*)=0,\\
        &(ii)\ \ \gamma_kf_k(x^*)=0\ \ k\in\{1,...,m\},\\
        &(iii)\ A_J+\sum\limits_{k=1}^m\gamma_kA_k \text{ is non negative}
    \end{split}
\end{equation}
are necessary for optimality. Given a feasible $x^{*}$ for problem \ref{QCQP1}, if \eqref{slater} holds, the conditions \eqref{kkt} are also sufficient for global optimality of $x^*$.
\end{theorem}
\begin{proof}
Let $f_0(x):=J(x)-J(x^*)$. Since $x^*$ is a global minimizer of \ref{QCQP1}, $f_0(x)\ge0$ $\forall\,x$ feasible for \ref{QCQP1}.
Hence, the system $f_k(x)<0$ $k=0,...,m$ has no solution.
By \autoref{th:shilbert}, there exists  $(\gamma_0,...,\gamma_m)\in\mathbb{R}^{m+1}_+\backslash 0$ such that 
\begin{equation} \label{alternative}
\gamma_0f_0(x)+\sum\limits_{k=1}^m\gamma_kf_k(x)\ge0\ \ \text{for all }x\in H.
\end{equation}
In particular, for $x=x^*$, we have $\sum\limits_{k=1}^m\gamma_kf_k(x^*)\ge0$. 
Since $\gamma_kf_k(x^*)\le0$ $\forall\,k\in \{1,...,m\}$, it must be $\gamma_kf_k(x^*)=0$ $\forall\,k\in \{1,...,m\}$ which proves (ii) of \eqref{fritzjohn}.
Moreover, by \eqref{alternative} and (ii) of \eqref{fritzjohn}, for all $x\in H$
\begin{equation}
    \label{attmin}
        \Tilde{L}(x):=\gamma_0J(x)+\sum\limits_{k=1}^m\gamma_kf_k(x)\ge \gamma_0J(x^*)=\Tilde{L}(x^*).
\end{equation}
Hence $\Tilde{L}(x)$ attains its minimum over $H$ at $x^*$.
Notice that we can rewrite $\Tilde{L}(x)$ as
\begin{equation}
\label{fjqcqp}
\Tilde{L}(x):=\gamma_0J(x)+\sum\limits_{k=1}^m\gamma_kf_k(x)=\langle x,\Tilde{A}(\gamma)x\rangle+2\langle \Tilde{b}(\gamma),x\rangle+\tilde{c}(\gamma)
\end{equation}
with $\Tilde{A}(\gamma)=\gamma_0A_J+\sum\limits_{k=1}^m\gamma_kA_k$,  $\Tilde{b}(\gamma)=\gamma_0b_J+\sum\limits_{k=1}^m\gamma_kb_k$ and  $\Tilde{c}(\gamma)=\gamma_0c_J+\sum\limits_{k=1}^m\gamma_kc_k$.

The operator $\Tilde{A}(\gamma)$ must be non negative, otherwise we would have $min_{x\in H}\,\Tilde{L}(x)=-\infty$, which is contradiction with \eqref{attmin}. 
Hence, $\Tilde{A}(\gamma)$ is non negative, i.e.  condition (iii) of \eqref{fritzjohn} holds and $\Tilde{L}(x)$ is convex with respect to $x$.
We can apply \autoref{fermat} for the convex and twice continuously differentiable function $\Tilde{L}(x)$. The optimality condition $\nabla_x\Tilde{L}(x^*)=0$ is equivalent to 
the conditions (i) of \eqref{fritzjohn}. 
The proof of the first part of the theorem is finished.

Suppose now that \eqref{slater} holds, i.e., there exists a point $x_0$ such that ${f_k(x_0)<0}$ $\forall\,k=1,...,m$. 
If it were $\gamma_0=0$, then by \eqref{alternative}, it would be $\sum\limits_{k=1}^m\gamma_kf_k(x)\ge0$ for all $x\in\mathbb{R}^n$, which would contradict \eqref{slater}. Hence $\gamma_0>0$ and the Fritz-John conditions becomes the KKT condition, i.e. \eqref{kkt} holds.

To complete the proof, we show that conditions \eqref{kkt} are also sufficient for optimality. 
Assume that there exists  $x^*\in H$  which is feasible to \ref{QCQP1} and $(\gamma_1,...,\gamma_m)\in\mathbb{R}^m_+\backslash0$ such that \eqref{kkt} holds.
The Lagrangian for \ref{QCQP1} is:
\begin{equation}
\label{lagqcqp}
L(x,\gamma):=J(x)+\sum\limits_{k=1}^m\gamma_kf_k(x)=\langle x,A(\gamma)x\rangle+2\langle b(\gamma),x\rangle+c(\gamma)
\end{equation}
with $A(\gamma)=A_J+\sum\limits_{k=1}^m\gamma_kA_k$,  $b(\gamma)=b_J+\sum\limits_{k=1}^m\gamma_kb_k$ and  $c(\gamma)=c_J+\sum\limits_{k=1}^m\gamma_kc_k$.

Notice that the Lagrangian $L(x,\gamma)$ is convex with respect to $x$, since  $A(\gamma)=A_J+\sum\limits_{k=1}^m\gamma_kA_k$ is non negative by \eqref{kkt}. 
Hence $x^*$ such that ${\nabla_x L(x^*)=0}$ is the minimum of $L(x,\gamma)$ for $\gamma$ fixed, by \autoref{fermat}.

By \eqref{kkt}, $\gamma_kf_k(x^*)=0$ for $k=1,...,m$. We have
\begin{equation}
\label{lagmin}
J(x)+\sum\limits_{k=1}^m\gamma_kf_k(x)\ge J(x^*)+\sum\limits_{k=1}^m\gamma_kf_k(x^*)=J(x^*)\ \ \forall\,x\in\mathbb{R}^n.
\end{equation}
For any $x$ feasible for \ref{QCQP1}, $f_k(x)\le0$, $k+1,...,m$, and hence 
\begin{equation}
    \label{sowehave}
    J(x)\ge J(x)+\sum\limits_{k=1}^m\gamma_kf_k(x)
\end{equation}
Combining \eqref{lagmin} and \eqref{sowehave}, for any $x$ feasible for \ref{QCQP1}, we have
\begin{equation}
    \label{xglobmin}
    J(x)\ge J(x^*)
\end{equation}
which proves that $x^*$ is a global minimum for \ref{QCQP1}.

\end{proof} 

\begin{example}
Consider a convex \eqref{QCQP1} with $m$ constraints and let $x^*$ be the global minimum.
It is possible to show that \ref{GIS} is convex. 
Fix $\lambda\in(0,1]$. Take $v:=(v_0,...,v_m)\in \text{\ref{GIS}}$, $w:=(w_0,...,w_m)\in\text{\ref{GIS}}$,  i.e. there exists $x_v,x_w\in H$ such that
\begin{equation}
    \label{eq: conv_vw}
    f_k(x_v)<v_k,\quad f_k(x_w)<w_k,\quad \forall\,k\in\{0,...,m\}.
\end{equation}
We show that $\lambda v+(1-\lambda)w\in \text{\ref{GIS}}$, i.e. there exists $\Tilde{x}\in H$ such that
\begin{equation}
    \label{eq: conv_proof}
    f_k(\Tilde{x})\le \lambda f_k( x_v)+(1-\lambda)f_k(x_w)<\lambda v_k+(1-\lambda)w_k,\quad \forall\,k\in\{0,...,m\},
    \end{equation}
where the strict inequality is a consequence of \eqref{eq: conv_vw}.
We choose ${\Tilde{x}=\lambda x_v+(1-\lambda)x_w}$.
By the convexity of functions $f_k(\cdot)$, for all $k\in\{0,...,m\}$
$$
f_k(\Tilde{x})=f_k(\lambda x_v+(1-\lambda)x_w)\le \lambda f_k( x_v)+(1-\lambda)f_k(x_w). 
$$
So, \eqref{eq: conv_proof} holds and \ref{GIS} is convex. 
By Theorem \ref{globalminchar}, there exists a vector $(\gamma_0,...,\gamma_m)\in\mathbb{R}^{m+1}_+\backslash\ 0$ such that the conditions \eqref{fritzjohn} are necessary for optimality. 
Moreover, if there exists a point $x_0\in H$ such that \eqref{slater} holds, then there exists a vector $(\gamma_1,...,\gamma_m)\in\mathbb{R}^m_+\backslash0$ such that
\eqref{kkt} are necessary and sufficient for global optimality.\qed
\end{example}

\begin{example}
Consider problem \eqref{QCQP1} with one constraint, $m=1$, defined on a real Hilbert space $H$ of dimension $\text{dim } H$ such that $3\le\text{dim } H\le\infty$. Assume that there exists $\gamma_0,\gamma_1\in\mathbb{R}$ such that $(\gamma_0A_0+\gamma_1A_1)$ is positive.
Then, \ref{GIS} is convex by \cite{continopolyak}, Theorem 4.1. If there exists a global minimum $x^*$ of the considered \eqref{QCQP1} with $m=1$, we can apply Theorem \ref{globalminchar}, i.e. there exists a vector $(\gamma_0,\gamma_1)\in\mathbb{R}^{2}_+\backslash\ 0$ such that the conditions \eqref{fritzjohn} are necessary for optimality. 
Moreover, if there exists a point $x_0\in H$ such that \eqref{slater} holds, then there exists a scalar $\gamma_1>0$ such that
\eqref{kkt} are necessary and sufficient for global optimality.\qed
\end{example}

\begin{remark}
In Theorem \ref{globalminchar}, we use \eqref{slater} as a "constraint qualification", in the sense that \eqref{slater} allow us to take $\gamma_0>0$ in \eqref{fritzjohn}.
In analogy with convex nonlinear programming, it is possible to replace \eqref{slater} with other constraint qualification. 
For example, consider convex \ref{QCQP1} with linear constraints only (Convex QP). 
By Theorem 3.118 of \cite{bonnans2013perturbation}, if $x_1$ is a local solution of a Convex QP, there exists a vector $(\gamma_1,...,\gamma_m)\in\mathbb{R}_{+}^m\setminus0$ such that conditions $(i)$ and $(ii)$ of \eqref{kkt} hold in $x_1$. 
Since $(iii)$ of \eqref{kkt} holds automatically, Theorem \ref{globalminchar} can be proved even without assuming \eqref{slater} for Convex QP. \qed
\end{remark}
\section{Global minima characterization for (S-QCQP)}
\label{global:characterisation1}

In the present section we use the results of Section \ref{global:characterisation} to provide \eqref{kkt} characterization of global minima for \ref{QCQP2}. The main result of this section is Theorem \ref{omega0isconvex}.

Let $x^*$ be the global minimum of \ref{QCQP2}.
As in \autoref{globalminchar}, we use the notation
 $$
 f_0(x):=J(x)-J(x^*)=a_J\|x\|^2+\langle b_J,x\rangle-a_J\|x^*\|^2-2\langle b_J,x^*\rangle
 $$ 

Also in the case of \ref{QCQP2}, \ref{GIS} takes the form
\begin{equation*}
    \Omega_0:=\{(f_0(x),f_1(x),..,f_m(x))|x\in H\}+ int\mathbb{R}^{m+1}_+
\end{equation*}
where 
$$
f_k(x):=a_k\|x\|^2+2\langle b_k,x\rangle+c_k\ \ k=0,...,m
$$
with $a_k\in\mathbb{R}$.

\begin{theorem}
\label{omega0isconvex}
Consider problem \ref{QCQP2} with $m$ constraints. Let the Hilbert space $H$ be such that one of the following holds:
\begin{enumerate}
    \item H is infinite dimensional,
    \item H has $n$-dimensional, $n\in\mathbb{N}$. In this case, any maximally linearly independent subset of the system $\langle b_k,x\rangle=0$, $k=0,...,m$ has cardinality $\Bar{m}<n$. 
\end{enumerate}
Then, \ref{GIS}, denoted with $\Omega_0$, is convex.
\end{theorem}

\begin{proof}
     In order to show that $\Omega_0\subset\mathbb{R}^{m+1}$ is convex, take any ${v:=(v_0,...,v_m)}$, ${w:=(w_0,...,w_m)\in \Omega_0}$ and $\lambda\in(0,1)$. 
There exist $x_v,x_w\in H$ such that
\begin{equation}
\label{eeeh_blue}
    f_k(x_v)<v_k\ \ \text{and }
   f_k(x_w)<w_k\ \ \forall\,k\in\{0,...,m\}.
\end{equation}

Consider the convex combination $\lambda v + (1-\lambda)w$. Let $(\lambda v + (1-\lambda)w)_k$ be the $k$-th component of $\lambda v + (1-\lambda)w$.
By \eqref{eeeh_blue}, we have 
\begin{equation}
    \label{fromlesstole_blue}
    \lambda f_k(x_v)+(1-\lambda)f_k(x_w)<(\lambda v + (1-\lambda)w)_k\ \ \forall\,k\in\{0,...,m\}.
\end{equation}

In order to prove that $\Omega_0$ is convex, we show that the convex combination
$\lambda v + (1-\lambda)w$  belongs to $\Omega_0$, i.e. there exists $\Tilde{x}$ such that $\forall\,k\in\{0,...,m\}$
\begin{equation}
    \label{needtoprove_blue}
    f_k(\Tilde{x})<(\lambda v + (1-\lambda)w)_k.
\end{equation}

Formulas \eqref{fromlesstole_blue} and \eqref{needtoprove_blue} together imply that $\Omega_0$ is convex if there exists $\Tilde{x}\in H$ such that $\forall\,k\in\{0,...,m\}$
\begin{equation}
    \label{needtoprove2_blue}
    f_k(\Tilde{x})\le\lambda f_k(x_v)+(1-\lambda)f_k(x_w).
\end{equation}

Note that if $x_v=x_w$, then \eqref{needtoprove2_blue} trivially holds for $\Tilde{x}=x_v=x_w$. From now on, we assume that $x_{v}$ and $x_{w}$ are distinct vectors, $x_v\neq x_w$.

Let us consider
\begin{equation}
    \label{equality_blue}
    \mathcal{S}(\lambda):=\{x\in H\ |\ \|x\|^2=\lambda \|x_v\|^2 + (1-\lambda)\|x_w\|^2 \},
\end{equation}
i.e. a sphere centered at zero with radius $\sqrt{\lambda \|x_v\|^2 + (1-\lambda)\|x_w\|^2}$ . In particular, when both $x_{v}=x_{w}=0$ the set $\mathcal{S}(\lambda)$ reduces to $\{0\}$, but this is impossible, since we assumed that $x_{v}$ and $x_{w}$ are distinct vectors.

 The idea of proving the existence of $\tilde{x}$ satisfying \eqref{needtoprove_blue} is based on the following observations:
\begin{description}
\item[--] any $\Tilde{x}\in\mathcal{S}(\lambda)$ satisfies \eqref{needtoprove_blue} if $\forall\,k\in\{0,...,m\}$:
\begin{equation}
\label{eq23_blue}
   \begin{split}
   &f_k(\Tilde{x})\le\lambda f_k(x_v)+(1-\lambda)f_k(x_w)<\lambda v_k+(1-\lambda)w_k.\\
    &a_k \|\Tilde{x}\|^2+2\langle b_k,\Tilde{x}\rangle+c_k\le a_k(\lambda \|x_v\|^2 + (1-\lambda)\|x_w\|^2)+2\langle b_k,(\lambda x_v + (1-\lambda)x_w)\rangle+c_k,
   \end{split}
\end{equation}
\item[--] to ensure \eqref{eq23_blue}, we need to show that we can choose $\Tilde{x}\in\mathcal{S}(\lambda)$ satisfying \eqref{equality_blue} and such that $\forall\,k\in\{0,...,m\}$

\begin{equation}
\label{syyy_blue}
   \begin{split}
       & \langle b_k,\Tilde{x}\rangle\le  \langle b_k,(\lambda x_v + (1-\lambda)x_w)\rangle,\ i.e.,\\
       &\langle b_k,(\Tilde{x}-(\lambda x_v + (1-\lambda)x_w))\rangle\le0.
   \end{split}
\end{equation}
\end{description}
In \eqref{syyy_blue}, we replace the inequality with the equality. We get
\begin{equation}
\label{equality_syst_blue}
       \langle b_k,(\Tilde{x}-(\lambda x_v + (1-\lambda)x_w))\rangle=0\ \forall\,k\in\{0,...,m\}.
\end{equation}

Let $V$ be the set of the solutions of \eqref{equality_syst_blue}.

Assume that 
\begin{equation} 
\label{sphere_blue} 
V\cap\mathcal{S}(\lambda)\neq\emptyset
\end{equation}
and take $\Tilde{x}\in V\cap\mathcal{S}(\lambda) $.
Then, we have 
for all $k\in\{0,...,m\}$
\begin{equation}
    \label{endproof1_blue1}
    \begin{split}
        &f_k(\Tilde{x})=a_k \|\Tilde{x}\|^2+2\langle b_k,\Tilde{x}\rangle+c_k=\\
        &a_k(\lambda \|x_v\|^2 + (1-\lambda)\|x_w\|^2)+2\langle{b_k},(\lambda x_v + (1-\lambda)x_w)\rangle+c_k=\\
       &\lambda( a_k \|x_v\|^2+2\langle b_k,x_v\rangle+c_k)+(1-\lambda) (a_k \|x_w\|^2+2\langle b_k,x_w\rangle+c_k)=\\
       &\lambda f_k(x_v)+(1-\lambda) f_k(x_w),
    \end{split}
\end{equation}
which proves Theorem \ref{omega0isconvex}.
Hence, in order to finish the proof, we need to show that \eqref{sphere_blue} holds, i.e. $V\cap\mathcal{S}(\lambda)\neq\emptyset$.
We change variable by setting 
\begin{equation}
    \label{change_var}
    y=\Tilde{x}-(\lambda x_v + (1-\lambda)x_w).
\end{equation}

The sets $\mathcal{S}(\lambda)$ and $V$ become
\begin{equation}
    \label{real_system}
    \left\{\begin{aligned}
        &\mathcal{S}(\lambda):=\{y\in H\ |\ \|y+(\lambda x_v + (1-\lambda)x_w)\|^2=\lambda \|x_v\|^2 + (1-\lambda)\|x_w\|^2 \},\\
        &V:=\text{ solution set of } \langle b_k,y\rangle=0\ \ \forall\,k\in\{0,...,m\}.
    \end{aligned}\right.
\end{equation}
Notice that, in the space of the variable $y\in H$, $\mathcal{S}(\lambda)$ is a sphere centered in $(-\lambda x_v - (1-\lambda)x_w)$ of radius
$\sqrt{\|x_v\|^2 + (1-\lambda)\|x_w\|^2}$, while $V$ is the space of solutions of the homogeneous system of equations of the form
\begin{equation}
    \label{linsystem_blue}
        \langle b_k,y \rangle = 0\ \ k\in\{0,...,m\}
\end{equation}

We choose a maximally linearly independent
subset of cardinality $\Bar{m}$ of the vectors $b_k$, $k\in\{0,...,m\}$. Without loss of generality, we will indicate the vectors in the maximally linearly independent subset with  $b_k$, $k\in\{1,...,\bar{m}\}$.
Then, the solution set $V$ of system \eqref{linsystem_blue}, coincides with the kernel of the linear continuous operator $\mathcal{B} : H \rightarrow \mathbb{R}^{\bar{m}}$, defined as
\begin{equation}
    \label{kernel_linsyst}
    \mathcal{B}y:=\begin{bmatrix}&\langle b_1,y\rangle\\ &\vdots \\ & \langle b_{\bar{m}},y\rangle
    \end{bmatrix}.
\end{equation}

The operators is continuous because, by the Riesz theorem, each functional $\langle b_k, y\rangle$, $k=1,...,\bar{m}$ is continuous on $H$.
Notice that the range of the operator $\mathcal{B}$ is finite-dimensional, i.e.,
$\text{range}(\mathcal{B}) = \mathbb{R}^{\Bar{m}}$.

We can write $\text{ker }\mathcal{B}=V= \bigcap\limits_{k=1}^{\bar{m}}\{x \in H\ |\ \langle b_k, x\rangle = 0\},$ 
for the linearly independent set $\{b_l,b_2, \cdots,b_{\bar{m}} \}$ in $H$.
By \eqref{eq:range+ker}, we have
\begin{equation}
    \label{eq:range+kerB}
    H=\text{ker }\mathcal{B}\oplus\text{range }\mathcal{B}^*=V\oplus\text{range }\mathcal{B}^*.
\end{equation}

By Theorem \ref{th:deu712}, $V$ is a subspace of codimension $\bar{m}$. 
By the definition of codimension, $\text{dim range }\mathcal{B}^*=\bar{m}$.  

We prove by contradiction that $V\neq\{0\}$, in the case that $H$ is infinite dimensional, assumption $i.$, and in the case where $H$ is finite dimensional with $n>\bar{m}$, assumption $ii.$

By \eqref{eq:range+kerB}, if $V=\{0\}$
, $\forall\,y\in H$ and for $0\in H$, there exists $\bar{y}\in\text{range }\mathcal{B}^* $ such that $y=\bar{y}+0$.
Then a basis of $\mathcal{B}^*$, which consist in a set of $\bar{m}$ maximally linearly independent vectors by Definition \ref{def:Schauderbasis}, is also a basis for $H$. This is a contradiction with $V=\{0\}$ in case that $H$ is infinite dimensional, since $\bar{m}<\infty$, see Example 2.27 of \cite{bauschke2017convex}, and also in case $H$ is $n$ dimensional, because $\bar{m}<n$, by assumption.

Hence, there exists a point $y_v\in V$ of norm $\|y_v\|>0$. 

We can prove that there exists $\Tilde{y}\in V\cap \mathcal{S}(\lambda)$ analitically, as follows.
We want to find $\alpha\in\mathbb{R}$ such that $\Tilde{y}:=\alpha y_v\in V\cap \mathcal{S}(\lambda)$.

We always have $\Tilde{y}:=\alpha y_v\in V$, since, $\forall\, k\in\{1,...,\Bar{m}\}$, ${\langle b_k, \Tilde{y}\rangle =\alpha\langle b_k, y_v\rangle = 0}$, in fact $V$ is a subspace and every subspace is a cone.

We recall that $\mathcal{S}(\lambda)$ is the sphere defined in \eqref{real_system}. Hence, in order to prove that $\Tilde{y}\in \mathcal{S}(\lambda)$, there should exists $\alpha\in\mathbb{R}$ such that

\begin{equation}
\label{provethis_blue1}
\begin{split}
&\Tilde{y}=\alpha y_v\in\mathcal{S}(\lambda),\ i.e.\\
    &\|\alpha y_v+\lambda x_v+(1-\lambda)x_w\|^2=\lambda \|x_v\|^2 + (1-\lambda)\|x_w\|^2.
\end{split}
\end{equation}
Consequently, we need to find $\alpha\in\mathbb{R}$ such that
\begin{equation}
\label{provethis2_blue1}
    \begin{split}
       &\|\alpha y_v+\lambda x_v+(1-\lambda)x_w\|^2=\\
    &=\alpha^2\| y_v\|^2+\|\lambda x_v+(1-\lambda)x_w\|^2+2\alpha\langle y_v,\lambda x_v+(1-\lambda)x_w\rangle=\\
    &=\lambda \|x_v\|^2 + (1-\lambda)\|x_w\|^2  
    \end{split}
\end{equation}
By Corollary 2.15 of \cite{bauschke2017convex}, 
\begin{equation}
    \label{corollary214_blue1}
    \|(\lambda x_v + (1-\lambda)x_w)\|^2=\lambda\|x_v\|^2 +(1-\lambda)\|x_w\|^2-\lambda(1-\lambda)\|x_v-x_w\|^2.
\end{equation}
Hence, \eqref{provethis2_blue1} becomes
\begin{equation}
    \label{endproof}
    \alpha^2\|y_v\|^2+2\alpha\langle y_v,\lambda x_v+(1-\lambda)x_w\rangle-\lambda(1-\lambda)\|x_v-x_w\|^2=0.
\end{equation}
Note that by \eqref{endproof} it must be $\alpha\neq0$ since $x_v$ and $x_w$ are distinct vectors.
There exists $\alpha$ such that \eqref{provethis_blue1} is satisfied (i.e. $\Tilde{x}\in\mathcal{S}(\lambda)$), if and only if

\begin{align*}
    &\Delta:=(\langle y_v, (\lambda x_v + (1-\lambda)x_w)\rangle)^2+\lambda(1-\lambda)\|y_v\|^2\|x_v-x_w\|^2\ge0,
\end{align*}

which holds for every $x_v,x_w\in H$.

\end{proof}

\begin{figure}[H]
    \includegraphics[width=8cm]{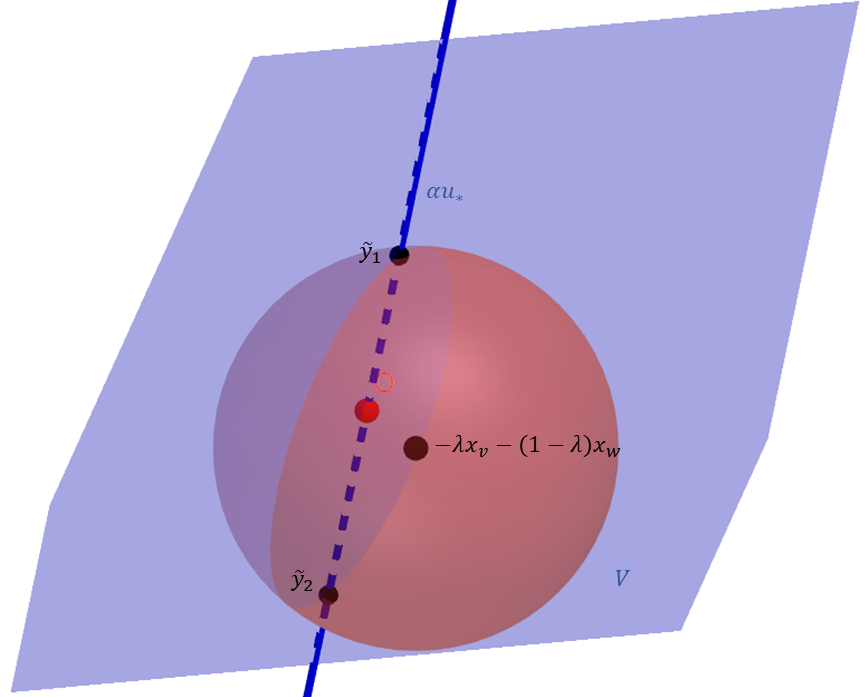}
    \caption{3d representation of \eqref{sphere_blue}. The sphere $\mathcal{S}(\lambda)$ in red is centered in ${-(\lambda(x_v)-(1-\lambda)x_w)}$. $V$ is in light blue, while the line $\alpha y_v$ is in dark blue. $\mathcal{S}(\lambda)$ intersects $\alpha y_v$ in the points $\tilde{y}_1$ and $\tilde{y}_2$. $0$ is in red. The picture was built with the 3D calculator of Geogebra, \cite{gg}. }
    \label{figure1}
\end{figure}

\begin{remark}
For $H=\mathbb{R}^n$, assumption $ii.$ of Theorem \ref{omega0isconvex} is automatically satisfied if $m+1<n$. This assumption is close to $m+1\le n$. For $H=\mathbb{R}^n$ and $m+1\le n$, Theorem \ref{omega0isconvex} can be considered a special case of Theorem 2.3 from \cite{beck2007convexity}.
\end{remark}

\begin{remark}
\begin{enumerate}
\item In \ref{figure1}, we provide a geometrical representation of the sets defined in \eqref{real_system}, in the space of the variable $y\in\mathbb{R}^3$, with $m=1$ such that $m+1=2<n=3$.
In $\mathbb{R}^3$, the above prove shows that the sphere $\mathcal{S}(\lambda)$ centered in ${-(\lambda(x_v)+(1-\lambda)x_w)}$ of radius ${\sqrt{\lambda \|x_v\|^2 + (1-\lambda)\|x_w\|^2}}$, always intersects the line $\alpha u_*$ of parameter $\alpha$, where $u_*$ belongs to a basis of the vector space $V$. 
In fact,
\begin{itemize}
    \item the distance between $0$ and ${-(\lambda(x_v)+(1-\lambda)x_w)}$ is always less than the radius of $\mathcal{S}(\lambda)$ by \eqref{corollary214_blue1},
    \item the line $\alpha u_*$ passes through $0$.
\end{itemize}
\item Observe that the assumption $\Bar{m}<n$ is important for the validity of the presented proof of Theorem \ref{omega0isconvex} in finite dimensions. Otherwise,  the only solution of the system \eqref{linsystem_blue} is $y=0$ which implies that it must be $\tilde{x}=\lambda x_{v}+(1-\lambda)x_{w}\in\mathcal{S}(\lambda)$, i.e.  $\lambda x_{v}\perp (1-\lambda)x_{w}$ which can hardly be satisfied.

\item Under the assumption of \autoref{omega0isconvex}, there could exist a component $i\in\{0,...,m\}$ such that $f_i(x):=\|x\|^2$ and a component $j\in\{0,...,m\}$ such that $f_j(x):=-\|x\|^2$.

In this case, $\Tilde{x}$ must satisfies
\begin{equation*}
    \begin{cases}
     &\|\Tilde{x}\|^2\le \lambda \|x_v\|^2 + (1-\lambda)\|x_w\|^2\\
     &-\|\Tilde{x}\|^2\le -(\lambda \|x_v\|^2 + (1-\lambda)\|x_w\|^2)
    \end{cases}
\end{equation*}

The above proves that, in order to complete the proof of  Theorem \ref{omega0isconvex}, we cannot choose $\Tilde{x}$
such that 
\begin{equation*}
    \|\Tilde{x}\|^2\neq\lambda \|x_v\|^2 + (1-\lambda)\|x_w\|^2
\end{equation*}
   This motivates our approach of looking for suitable $\tilde{x}$ from among elements of $\mathcal{S}(\lambda)$. 

\end{enumerate}
\end{remark}

Theorem \ref{omega0isconvex} allows us to prove that \eqref{kkt} conditions are necessary and sufficient optimality conditions for  \ref{QCQP2}, as stated in the following theorem.

\begin{theorem}
\label{sqcqpcharcter}
Consider \ref{QCQP2} and let the assumption of Theorem \eqref{omega0isconvex} be satisfied. 
Let \autoref{preliminaryassumpt} holds and
$x^*$ be a global minimizer of \ref{QCQP2}.
Then the Fritz-John conditions \eqref{fritzjohn} are necessary for optimality.
Moreover, if there exists a point $x_0\in H$ such that 
\begin{equation*} 
f_k(x_0)<0\ \ \ \forall\,k\in\{1,...,m\},
\end{equation*}
then, for a feasible $x^*$, the conditions
\eqref{kkt} take the form:
there exists a vector $(\gamma_1,...,\gamma_m)\in\mathbb{R}^m_+\backslash0$ such that
\begin{equation}
    \label{kkt2}
    \tag{KKT}
    \begin{split}
        &(i)\ \ \
        2a_Jx^*+b_J+2\sum\limits_{k=1}^m 2a_kx^*+b_k=0,\\
        &(ii)\ \ \gamma_kf_k(x^*)=0\ \ k\in\{1,...,m\},\\
        &(iii)\ \ \ a_J+\sum\limits_{k=1}^m\gamma_ka_k\ge0,
    \end{split}
\end{equation}
are necessary and sufficient for global optimality of $x^{*}$.

\end{theorem}

\begin{proof}
By \autoref{omega0isconvex}, \ref{GIS} is convex. 
We can apply \autoref{globalminchar} to complete the proof.
     $A_J+\sum\limits_{k=1}^m\gamma_kA_k$ can be rewritten as $(a_J+\sum\limits_{k=1}^m\gamma_ka_k)I$ when \autoref{assumaI} holds. Then, condition (iii) of  \eqref{kkt} becomes
 $a_J+\sum\limits_{k=1}^m\gamma_ka_k\ge0$ for \ref{QCQP2}. 

\end{proof}

\begin{remark} Let $A:H\rightarrow H$ be a continuous linear operator. 
Let the corresponding quadratic form $\langle x,Ax\rangle$ be positive definite, i.e. there exists $\alpha>0$ such that, $\forall\,x\in H$, $\langle x,Ax\rangle\ge\alpha\|x\|^2$, see \cite{hestenes1951applications}.
Then the function $\|\cdot\|_{A}:H\rightarrow\mathbb{R}_{+}$ defined as $\|x\|_{A}=\sqrt{\langle x,Ax\rangle}$ is a norm in $H$ (see \cite{bonnans2013perturbation}, section 3.3.2). 
With this observation, the characterisation of Theorem \ref{sqcqpcharcter} holds for scalar QCQP problems satisfying the following assumption.

\begin{assumption}
\label{assumaI1}
Let $A:H\rightarrow H$ be a continuous linear operator such that $\langle x,Ax \rangle$ is positive definite. 
Let $A_J,A_{k}$ $\forall\,k\in\{1,...,m\}$ be of the form $A_J=a_JA$, $A_k=a_kA$, where $a_J,a_k\in\mathbb{R}$. Then, $A$ induce the norm $\|\cdot\|_{A}$.
\end{assumption}
\end{remark}

Conditions for a positive quadratic form to be positive definite can be found in Theorem 10.1 of \cite{hestenes1951applications}.

\section{Conclusions}
According to Theorem \ref{globalminchar}, the convexity of \ref{GIS} and the Slater-type constraint qualification for a given QCQP suffice to characterize its global optimality in terms of KKT-type conditions. It is an open problem what is the largest class of QCQP problems for which such characterization holds true.

\section*{acknowledgements}
This project has received funding from the European Union's Horizon 2020 research and innovation programme under the Marie Sk{\l}odowska-Curie grant agreement No 861137.

\bibliographystyle{plain}
\bibliography{BedBrucArx.bib}  

\end{document}